\documentclass{amsart}

\usepackage{amsmath,amssymb,amsthm}
\usepackage{hyperref}
\usepackage{xcolor}
\usepackage{tikz-cd}

\newtheorem{thmintr}{Theorem}
\newtheorem{prop}{Proposition}[section]

\theoremstyle{definition}

\newtheorem{rem}[prop]{Remark}

\newtheorem*{ack}{Acknowledgement}

\def\co{\colon\thinspace}

\newcommand{\C}{\mathbb C}

\newcommand{\rmd}{\mathrm d}

\newcommand{\F}{\mathbb F}

\newcommand{\rmi}{\mathrm i}

\newcommand{\MM}{\mathcal M}

\newcommand{\N}{\mathbb N}

\newcommand{\bfp}{\mathbf p}

\newcommand{\bfq}{\mathbf q}

\newcommand{\R}{\mathbb R}

\newcommand{\bfu}{\mathbf u}

\newcommand{\bfv}{\mathbf v}

\newcommand{\bfw}{\mathbf w}

\newcommand{\bfx}{\mathbf x}

\newcommand{\bfy}{\mathbf y}

\newcommand{\Z}{\mathbb Z}

\newcommand{\bfz}{\mathbf z}

\newcommand{\lra}{\longrightarrow}
\newcommand{\ra}{\rightarrow}

\DeclareMathOperator{\cp}{\mathrm{Cap}}

\DeclareMathOperator{\ev}{\mathrm{ev}}

\DeclareMathOperator{\fs}{\mathrm{FS}}

\DeclareMathOperator{\id}{\mathrm{id}}
\DeclareMathOperator{\im}{\mathrm{Im}}

\DeclareMathOperator{\Int}{\mathrm{Int}}

\DeclareMathOperator{\pt}{\mathrm{pt}}

\DeclareMathOperator{\st}{\mathrm{st}}

\begin{document}

\author{Johanna Bimmermann}
\author{Bernd Stratmann}
\author{Kai Zehmisch}
\address{Fakult\"at f\"ur Mathematik, Ruhr-Universit\"at Bochum,
Universit\"atsstra{\ss}e 150, D-44801 Bochum, Germany}
\email{Johanna.Bimmermann@rub.de}
\email{Bernd.X.Stratmann@rub.de}
\email{Kai.Zehmisch@rub.de}

\title[Fitting without fittings]{Fitting without fittings}

\date{09.06.2025}

\begin{abstract}
  We show that all symplectically aspherical fillings
  of the unit cotangent bundle of a given odd-dimensional sphere
  are diffeomorphic to the corresponding unit co-disc bundle.
  The concept of fittings previously introduced is not needed.
\end{abstract}

\subjclass[2010]{57R17; 32Q65, 53D35, 57R80}
\thanks{This research is part of a project in the SFB/TRR 191
{\it Symplectic Structures in Geometry, Algebra and Dynamics}, 
funded by the DFG}

\maketitle


\section{Introduction\label{sec:intro}}


A closed co-oriented contact manifold $(M,\xi)$
often arises as the convex contact type boundary of a compact
symplectic manifold $(W,\omega)$,
which is called a (strong) {\bf symplectic filling} of $(M,\xi)$.
This raises the question,
whether the symplectic filling $(W,\omega)$ of $(M,\xi)$
is unique.
As the blow up of a symplectic filling $(W,\omega)$ of $(M,\xi)$
is a symplectic filling of $(M,\xi)$,
we require the symplectic filling to be {\bf symplectically aspherical},
i.e.\ $W$ admits no spherical homology $2$-class
on which the symplectic form $\omega$ evaluates non-trivially.

Regarding the uniqueness question,
substantial knowledge has already been established.
For the $4$-dimensional case,
where requiring minimality
instead of symplectic asphericity is sufficient and
diffeomorphism classification addressing the symplectic structure is possible,
we refer to the introductions in \cite{bgz19,gkz23,gz26}.
In higher dimensions,
there are examples where uniqueness up to diffeomorphism holds:
Eliashberg--Floer--McDuff proved
that a symplectically aspherical filling
of the standard contact sphere is diffeomorphic to a ball,
see \cite[Theorem 1.5]{mcd91}.
For generalisations to subcritically Stein fillable contact manifolds,
or even Liouville split ones, we refer to the work of
Barth--Geiges--Zehmisch \cite{bgz19} and Zhou \cite{zhou23},
resp.
The first examples in the critical regime,
taking the total space
of the unit cotangent bundle over $T^n$,
were independently found by Bowden--Gironella--Moreno \cite{bgm22}
and Geiges--Kwon--Zehmisch \cite{gkz23}.
In \cite{gkz23},
instead of a single $n$-torus $T^n$,
more generally a product of $T^2$
with a product of unitary groups and spheres
is allowed.
If the base manifold is an odd-dimensional sphere,
partial uniqueness results were obtained by
Kwon--Zehmisch \cite{kz19} and
Kwon--Oba \cite{ko24}.
The aim of this note
is to remove all restrictions.
Our approach builds up on intermediate results from \cite{kz19}
and is inspired
by \cite{ko24}.

Let $S^{2d+1}$, $d\geq1$, be the unit sphere
$\partial D^{2d+2}\subset\C^{d+1}$
provided with the round metric.
Equip the total space of the cotangent bundle $T^*S^{2d+1}$
with the Liouville $1$-form $\lambda$ 
given by $\bfp\,\rmd\bfq$ in local coordinates.
Consider the induced contact structure $\xi$
on the unit cotangent bundle $M=ST^*S^{2d+1}$
together with the unit disc bundle
\[
(W_{\!\st},\omega_{\st})=
\big(DT^*S^{2d+1},\rmd\lambda\big)
\,,
\]
the so--called {\bf standard filling} of the contact manifold $(M,\xi)$.

\begin{thmintr}
\label{thm:uniquenesstheorem}
The underlying smooth structure $W$
of a symplectically aspherical filling $(W,\omega)$ of $(M,\xi)$
is diffeomorphic to $W_{\!\st}$.
\end{thmintr}

The present work is devoted to proving the theorem.









\section{From properness to diffeomorphism type\label{sec:topologicalpart}}

For fixed $d\geq1$
let $(W,\omega)$ be a symplectically aspherical filling of $(M,\xi)$.
In this section we will summarise intermediate results
about the diffeomorphism type of $W$ obtained in \cite{kz19}.
The notion of fittings is not involved.


\subsection{Characterisation via $h$-cobordism\label{subsec:charviahc}}

\cite[Proposition 2.1]{kz19} yields that
$W$ and $W_{\!\st}$ are diffeomorphic if and only if
$W$ is simply connected and
the singular homology groups satisfy
$H_kW=H_kW_{\!\st}$ for all $k$.
This result is based on the $h$-cobordism theorem
applied to an $h$-cobordism that is constructed
using $\chi(S^{2d+1})=0$.


\subsection{Capping and reduction\label{subsec:capandred}}

\cite[Lemma 3.1]{kz19} reduces the criterion recalled
in Section \ref{subsec:charviahc}
to surjectivity properties of a certain map
in homotopy and homology.
The relevant map is constructed as follows:
Provide $\C^{d+1}$ with coordinates
\[
\bfz=\big(z_0,\ldots,z_d\big)
\,,\quad
\bfz=\bfx+\rmi\bfy
\,,
\]
and denote the standard symplectic form on $\C^{d+1}$
by $\rmd\bfx\wedge\rmd\bfy$.
The Fubini--Study symplectic form on $\C P^d$
is denoted by $\omega_{\fs}$.
The graph of the {\bf anti-Hopf map}
\[
S^{2d+1}\lra\C^{d+1}\times\C P^d\,,
\quad\bfz\longmapsto(\bfz,[\bar{\bfz}])\,,
\]
defines a Lagrangian embedding of
$S^{2d+1}\subset\C^{d+1}$.
Here,
$S^{2d+1}\ra\C P^d$, $\bfz\mapsto [\bfz]$,
is the Hopf map and $\bfz\mapsto\bar{\bfz}$
indicates complex conjugation.

By Weinstein's neighbourhood theorem
the anti-Hopf map extends to a symplectic embedding
of a disc neighbourhood $D_{\delta}T^*S^{2d+1}$,
$\delta>0$, of $S^{2d+1}$ in $T^*S^{2d+1}$.
Therefore, $S_{\delta}T^*S^{2d+1}$ appears
as hypersurface $(M,\xi)$ of contact type in
$\C^{d+1}\times\C P^d$ provided with
$\rmd\bfx\wedge\rmd\bfy+\omega_{\fs}$.
Define a symplectic cap by setting
\[
\cp:= \C^{d+1}\times\C P^d\setminus
\Int\big(D_{\delta}T^*S^{2d+1}\big)
\]
equipped with the symplectic form
\[
\omega_{\cp}:=
\big(
\rmd\bfx\wedge\rmd\bfy+\omega_{\fs}
\big)|_{\cp}
\,.
\]
Gluing along contact type boundaries
yields a symplectic manifold
\[
(Z,\Omega):=
(W,\omega)\cup_{(M,\xi)}
\big(\!\cp,\omega_{\cp}\big)
\,.
\]
Let $a_0\in\C$ be a complex number
with sufficiently large absolut value,
so that $a_0\times\C^d\times\C P^d$
is a complex hypersurface in $Z$.
Now,
by \cite[Lemma 3.1]{kz19},
$W$ and $W_{\!\st}$ are diffeomorphic
provided that the inclusion map
$a_0\times\C^d\times\C P^d\ra Z$
is surjective after applying
$\pi_1$ and $H_k(\,.\,;\F)$ for all $k\in\Z$
and all fields $\F$.


\subsection{Splitting off a factor\label{subsec:splittoffafak}}

On $\C$, provided with coordinates $z_0=x_0+\rmi y_0$,
we consider the symplectic form $a(x_0,y_0)\rmd x_0\wedge\rmd y_0$
for a smooth function $a\co\C\ra(0,\infty)$
that equals $1$ on $B_r^2(0)$ for some $r>1$ and
\[
a(x_0,y_0)=\big(1+x_0^2+y_0^2\big)^{-2}
\]
for $(x_0,y_0)$ in the complement of a compact set in $\C$.
The symplectic form extends to a symplectic form $\sigma$
on $\C P^1$ via the affine chart $[1:z_0]\mapsto z_0$.
In particular,
$\sigma$ equals $\omega_{\fs}$ in a neighbourhood of $\infty=[0:1]$.
Observe that the total area $A$ of $\sigma$ is at least $\pi r^2$.

We split off the $z_0$-factor
$\C\times\C^d \times\C P^d$ of $\C^{d+1}\times\C P^d$.
Coordinates on the $\C^d$-factor are introduced as
$\bfz'=\bfx'+\rmi\bfy'=(z_1,\ldots,z_d)$.
We symplectically compactify partially to
\[
\Big(\C P^1\times\C^d \times\C P^d,
\sigma+\rmd\bfx'\!\wedge\rmd\bfy'+\omega_{\fs}\Big)
\,,
\]
which, as in Section \ref{subsec:capandred},
can be assumed to contain
$(M,\xi)$ as hypersurface of contact type.
A symplectic cap
\[
\widehat{\cp}:=
\C P^1\times\C^d\times\C P^d \setminus
\Int\big(D_{\delta}T^*S^{2d+1}\big)
\]
with symplectic form
\[
\hat{\omega}:=
\sigma+\rmd\bfx'\!\wedge\rmd\bfy'+\omega_{\fs}
\]
can be defined as in Section \ref{subsec:capandred},
as well as the glued symplectic manifold
\[
\big(\hat{Z},\hat{\Omega}\big):=
(W,\omega)\cup_{(M, \xi)}\Big(\widehat{\cp},\hat{\omega}\Big)
\,.
\]


\subsection{Complex cycles\label{subsec:complexcycles}}

The standard complex structure
on $\C P^1\times\C^{d}\times\C P^d$
is denoted by
$J_{\st}=\rmi\oplus\rmi\oplus\rmi$.
Consider the holomorphic spheres
\[
C_1:=\C P^1\times\pt\times\pt\,,
\qquad
C_2:=\pt\times\pt\times\C P^1\,,
\]
in $\widehat{\cp}$,
where $\C P^1\subset \C P^d$ denotes a complex line
in the definition of $C_2$.
Define complex hypersurfaces
\[
H_1:=\infty\times\C^d\times\C P^d\,,
\qquad
H_2:=\C P^1\times\C^d\times\C P^{d-1}\,,
\]
where $\C P^{d-1}\subset\C P^d$
denotes a complex hyperplane.
Observe that $H_1\subset\widehat{\cp}$,
while $H_2$ is just a hypersurface in
$\C P^1\times\C^d \times\C P^d$.
We choose $\C P^1$ and $\C P^{d-1}$
to be in general position in $\C P^d$,
so that the intersection numbers
$C_i\cdot H_j=\delta_{ij}$
can be read off for $i,j=1,2$.


\subsection{Almost complex structure and maximum principle\label{subsec:almcpxstrandmp}}

Let $J$ be a tame almost complex structure
on $(\hat{Z},\hat{\Omega})$
that equals $J_{\st}$ on $\widehat{\cp}$,
cf.\ \cite[Section 3.3]{kz19}.
Consider a holomorphic sphere
$C=u(\C P^1)$ parametrised by $u\co\C P^1\ra(\hat{Z},J)$.

\begin{rem}
\label{rem:maximumprincipleA}
The maximum principle implies
that if $C\subset\widehat{\cp}$, then $u$
is of the form $u=(u^1,\pt,u^3)$, where $\pt\in\C^{d}$ and
$u^1\co\C P^1\ra\C P^1$, $u^3\co\C P^1\ra\C P^d$
are holomorphic, cf.\ \cite[Lemma 3.2(i)]{kz19}.
\end{rem}

\begin{rem}
\label{rem:maximumprincipleB}
Choose $R\in(1,r)$ large enough
such that $M$ appears as a hypersurface in
$B_R^2(0)\times B_R^{2d}(0)\times\C P^d$
and define
\[
W_{\square}:=W\cup
\Big\{(z_0,\bfz',\bfw)\in\widehat{\cp}\;\big|\;|z_0|,|\bfz'|<R\Big\}
\,.
\]
Assume $C$ satisfies
$C\cap\big(\hat{Z}\setminus W_{\square}\big)\neq\emptyset$
and $C\cap H_1=\emptyset$,
then the maximum principle in $\C$- and in $\C^d$-direction
applies as in \cite[Lemma 3.2(ii)]{kz19}.
Namely,
the $\C$- and $\C^d$-projections of $C$ must be
contained in the level sets of the radial distance functions,
so that $C$ in fact is contained in
$\hat{Z}\setminus W_{\square}\subset\widehat{\cp}$.
By Remark \ref{rem:maximumprincipleA} finally,
$u$ is of the form $u=(\pt,\pt,u^3)$,
where $u^3: \C P^1 \ra \C P^d$ is holomorphic.
\end{rem}


\subsection{Moduli space and reduction\label{subsec:modspaceandred}}

Define the {\bf moduli space} $\MM$
to be the set of all holomorphic maps
$u\co\C P^1\ra(\hat{Z},J)$
with $u(\C P^1)$ homologous to $C_1$,
mapping the marked points
$z\in\{\pm\varrho,\infty\}$ for fixed $\varrho\in(R,r)$
into $z\times\C^d\times\C P^d$.
In \cite[Remark 3.3]{kz19} the following was observed:

All $u\in\MM$ satisfy $u\bullet H_1=1$,
hence, are simple,
and have symplectic energy equal to $E(u)=A>\pi r^2$.
In particular,
if $u(\C P^1)\subset\widehat{\cp}$
for $u\in\MM$, then $u=(\id,\pt,\pt)$ by
Remark \ref{rem:maximumprincipleA},
so that $u$ is Fredholm regular.
Consequently,
$\MM$ is a smooth, naturally oriented
$4d$-dimensional manifold
by choosing $J$ on $\Int(W)$ to be generic
such that $u\in\MM$ is Fredholm regular
whenever $u(\C P^1)\cap\Int(W)\neq\emptyset$.

The maximum principle in $\C^d$-direction
as in Remark \ref{rem:maximumprincipleB} above,
shows that
  \[
   \Big\{u\in\MM\;\Big|\;u(\C P^1)\cap
   \big\{(z_0,\bfz',\bfw)\in\widehat{\cp}\;\big|\;|\bfz'|\geq R\big\}
   \neq\emptyset\Big\}
  \]
equals
  \[
   \big\{u=(\id,\bfz',\bfw)\;\big|\;\bfz'\in\C^d,\;|\bfz'|\geq R,\;\bfw\in\C P^d\big\}\,.
  \]
Therefore,
the end of $\MM$ can be identified with
$\big(\C^d\setminus B^{2d}_R(0)\big)\times\C P^d$.
The evaluation map
\[
\widehat{\ev}\co\MM\times\C P^1\lra\hat{Z}
\,,\quad
(u,z)\longmapsto u(z)
\]
restricted to the end equals
$\big((\id,\bfz',\bfw),z\big)\mapsto(z,\bfz',\bfw)$.

Reducing to the criterion in Section \ref{subsec:capandred},
\cite[Lemma 3.4]{kz19} shows:
If $\widehat{\ev}$ is proper,
then $W$ and $W_{\!\st}$ are diffeomorphic.


\subsection{Noded curves and reduction\label{subsec:nodedcandred}}

Define
\[
\hat{Z}_R:=W\cup
\Big\{(z_0,\bfz',\bfw)\in\widehat{\cp}\;\big|\;|\bfz'|\leq R\Big\}
\,.
\]
Uniform energy and $C^0$-bounds
discussed in Section \ref{subsec:modspaceandred}
yield as in \cite[Section 3.4.1]{kz19}
that any sequence $u^{\nu}\in\MM$ with
$u^{\nu}(\C P^1)\subset\hat{Z}_R$ for all $\nu\in\N$
admits a subsequence (again denoted by) $u^{\nu}$
that converges in the sense of Gromov compactness
to a stable holomorphic sphere $\bfu$
with three fixed marked points.
Denoting the non-constant components of $\bfu$
by $u_1,\dots,u_N$, $N\geq1$,
the convergence is in $C^{\infty}$,
meaning that the evaluation map $\widehat{\ev}$ is proper,
if and only if $N=1$.
In other words,
 \emph{if it holds that always $N=1$,
 then $W$ and $W_{\!\st}$ will be diffeomorphic.}

Invoking positivity of intersections
with the complex hypersurface $H_1$
and that $u^{\nu}\bullet H_1=1$ holds for all $\nu\in\N$
as in \cite[Section 3.4.2]{kz19} yields
(after potentially relabeling the bubble spheres that)
$u_1\bullet H_1=1$ and $u_j\bullet H_1=0$
for all $j\geq 2$.
Moreover, with Remark \ref{rem:maximumprincipleB},
we have  for each $j\geq 2$ that
either $u_j(\C P^1)\subset W_{\square}$
or $u_j=(\pt,\pt,u^3_j)$ with
$u^3_j: \C P^1 \ra \C P^d$ holomorphic.
Observe that $u_1(\C P^1)\subset\hat{Z}_R$
and that $u_j(\C P^1)\subset Z_R$ for all $j\geq 2$,
where
\[
Z_R:=W\cup
\Big\{(z_0,\bfz',\bfw)\in\widehat{\cp}\setminus H_1\big|\;|\bfz'|\leq R\Big\}
\,.
\]

Applying the Mayer--Vietoris sequence to the manifolds
$\hat{Z}=W\cup_M\!\widehat{\cp}$ and $Z=W\cup_M\!\cp$,
which are glued together along the simply connected hypersurface $M$,
yields a decomposition of spherical homology classes
\[
H_2\hat{Z}\ni
[u_1]=
S_1+A_1
\in H_2W\oplus H_2\widehat{\cp}
\]
and
\[
H_2Z\ni
[u_j]=
S_j+A_j
\in H_2W\oplus H_2\!\cp
\,,
\]
$j\geq 2$, resp.,
as in \cite[Section 3.4.3]{kz19}.
The respective classes $S_j$ and $A_j$
are sums of spherical classes for all $j=1,\dots,N$.
Observe that for $d\geq2$ the decompositions are unique;
for $d=1$ unique up to adding the fibre class
of $M\cong S^3\times S^2$.
Symplectic asphericity of $(W,\omega)$ yields
\[
\hat{\Omega}(S_j)=0
\quad
\text{and, hence,}
\quad
0<E(u_j)=\hat{\Omega}(A_j)
\]
for all $j=1,\dots,N$.
With the intersection pattern
from Section \ref{subsec:complexcycles}
and positivity of energy we find
as in \cite[Section 3.4.4]{kz19}
that there exists $\ell_1\in\Z$ and
$\ell_2,\ldots,\ell_N\in\N$ such that
\[
A_1=[C_1]+\ell_1[C_2]
\quad
\text{and}
\quad
A_j=\ell_j[C_2]\,,
\]
for all $j\geq 2$.

We will continue the argument at the beginning of
Section \ref{sec:asechypcla}.


\section{A second hyperplane class\label{sec:asechypcla}}

As the complex hypersurface $H_2$
of $\C P^1\times\C^d\times\C P^d$
is not a subset of $\widehat{\cp}$
the intersection pattern form
Section \ref{subsec:complexcycles}
can not be used to exclude bubbling off phenomena
in Section \ref{subsec:nodedcandred}.
In order to substitute $H_2$ we will consider
the complex hypersurface $\hat{H}$ of
$\C P^1\times\C^d\times\C P^d$
given by the zero set of the bihomogeneous polynomial
\[
z_0w_0+z_0'\big(z_1w_1+\ldots+z_dw_d\big)=0
\]
in Section \ref{subsec:projoffirstfak}.
In fact,
$\hat{H}$ is contained in the interior of $\widehat{\cp}$ and
has intersection numbers $C_i\cdot\hat{H}=1$ for $i=1,2$,
see Section \ref{subsec:projoffirstfak}.
Moreover,
in Section \ref{subsec:wwhomotopyequivalent}
we will isotope $W$ inside $\hat{Z}\setminus\hat{H}$
to a subset $W'$ that is symplectically aspherical
w.r.t.\ $\hat{\Omega}$ and
such that there exists an exhausting, strictly plurisubharmonic function
\[
\hat{h}\co\hat{Z}\setminus
\Big(
\hat{H}\cup\Int(W')
\Big)
\lra(0,\infty)
\]
for which $\partial W'$ appears as a level set,
see Section \ref{subsec:exapshf}.

With this preliminaries,
which we will prove in this section below,
we continue the argumentation from
Section \ref{subsec:nodedcandred}:

\begin{proof}[Proof of Theorem \ref{thm:uniquenesstheorem}]
  We take intersections of the bubble spheres $u_j$ with $\hat{H}$.
  Using affine charts for the composition of the map $u_j$
  with the dehomogenised polynomial,
  which defines $\hat{H}$ locally,
  we obtain with unique continuation the following:
  Either intersection points of $u_j(\C P^1)$ with $\hat{H}$ are isolated
  or $u_j(\C P^1)\subset\hat{H}$.
  In the case of isolated intersections
  the local intersection numbers are positive with \cite[p.~63]{gh78}.
  Note that if $u_1(\C P^1)$ is contained in $\widehat{\cp}$
  then $\ell_1\geq1$ by Remark \ref{rem:maximumprincipleA}.
  However,
  as $\hat{H}$ can be compactified via homogenisation
  as in Section \ref{subsec:apsvtypeemb},
  the sum of the local intersection numbers with $\hat{H}$
  is a homological invariant.
  Therefore,
  in either case, we get with the intersection numbers
  from Section \ref{subsec:projoffirstfak} that
  \[
  0\leq u_1\bullet\hat{H}=1+\ell_1
  \quad\text{and}\quad
  1\leq\ell_j=u_j\bullet\hat{H}
  \,,
  \]
  for all $j\geq 2$.
  Of course the positivity of the $\ell_j$ for $j\geq 2$
  was already observed in Section \ref{subsec:nodedcandred}.

  On the other hand $C_1=\sum_{j=1}^N[u_j]$
  and the expansions of the $A_j$
  at the end of Section \ref{subsec:nodedcandred} imply
  \[
  1=
  C_1\cdot\hat{H}=
  \sum_{j=1}^NA_j\cdot\hat{H}=
  (1+\ell_1)+\ell_2+\cdots+\ell_N\geq
  N-1\geq
  0
  \,.
  \]
  It follows that either $N=1$ with $\ell_1=0$
  or that $N=2$ and $\ell_1=-1$, $\ell_2=1$.
  The case $N=1$ means properness of $\widehat{\ev}$.
  In the case that $N=2$ we get
  \[
  A_1=[C_1]-[C_2]
  \quad\text{and}\quad
  A_2=[C_2]
  \,.
  \]
  In other words
  \[
  [u_1]=
  S_1+[C_1]-[C_2]
  \quad\text{and}\quad
  [u_2]=
  -S_1+[C_2]
  \,.
  \]
  The two potentially remaining bubble spheres
  have symplectic energy
  \[
  E(u_1)=A-\pi
  \quad\text{and}\quad
  E(u_2)=\pi
  \,.
  \]

  In order to exclude $N=2$
  observe that $\ell_1=-1$ implies $u_1\bullet\hat{H}=0$.
  As already remarked at the beginning of the proof,
  $u_1(\C P^1)\subset\hat{H}\subset\widehat{\cp}$
  would imply that $u_1^3$ represents $-[C_2]$ holomorphically,
  see Remark \ref{rem:maximumprincipleA}.
  We get that $u_1(\C P^1)$ and $\hat{H}$ are disjoint. 
  Evaluating $\hat{h}$ along the curve $u_1(\C P^1)$
  yields by the maximum principle that either
  $u_1(\C P^1)$ is contained in a level set of $\hat{h}$,
  so that $u_1$ is constant in view of the $-\rmd(\rmd\hat{h}\circ J)$--energy,
  or $u_1(\C P^1)$ is contained in $W'$,
  so that $u_1$ is constant in view of 
  symplectic asphericity in Section \ref{subsec:wwhomotopyequivalent}.
  In either case,
  we obtain a contradiction to the assumption
  that bubble spheres are not constant.
\end{proof}

In order to complete the proof of
Theorem \ref{thm:uniquenesstheorem}
we have to provide the constructions and claims
made at the beginning of this section.
This will occupy the rest of this section:


\subsection{Affine quadric\label{subsec:anaffquad}}

We provide $\C^{d+1}\times\C^{d+1}$
with coordinates
\[
(\bfz,\bfw)=\big(z_0,\ldots,z_d,w_0,\ldots,w_d\big)
\]
and consider the affine quadric
\[
Q:=\big\{\bfz\cdot\bfz+\bfw\cdot\bfw=1\big\}
\,.
\]
A decomposition
\[
\bfz=\bfx+\rmi\bfy
\,,\quad
\bfw=\bfu+\rmi\bfv
\,,
\]
in real and imaginary parts, yields
\[
Q=
\big\{
|\bfx|^2+|\bfu|^2=1+|\bfy|^2+|\bfv|^2
\,,\,\,
\bfx\cdot\bfy+\bfu\cdot\bfv=0
\big\}
\,.
\]
Consider the exhausting, strictly plurisubharmonic function
$f\co Q\ra[0,\infty)$ defined by restricting the map
$(\bfz,\bfw)\mapsto\frac12\big(|\bfy|^2+|\bfv|^2\big)$
to $Q$.
Observe that $f$ vanishes precisely
along the totally real $(2d+1)$-dimensional sphere
$\{|\bfx|^2+|\bfu|^2=1\}$.


\subsection{Transformed affine quadric\label{subsec:transaffquad}}

The linear coordinate change
\[
\Phi(\bfz,\bfw)=\big(\bfz+\rmi\bfw,\bfz-\rmi\bfw\big)
\]
maps the quadric $Q$
introduced in Section \ref{subsec:anaffquad} to
\[
\Phi(Q)=
\big\{
\bfz\cdot\bfw=1
\big\}
\,.
\]
The advantage of the transformed coordinates is that
$\Phi(\bfx,\bfu)=\big(\bfx+\rmi\bfu,\overline{\bfx+\rmi\bfu}\big)$
maps the real part $\{|\bfx|^2+|\bfu|^2=1\}$ of $Q$ onto
the totally real $(2d+1)$-sphere
\[
\Big\{
(\bfz,\bar{\bfz})\in\C^{d+1}\times\C^{d+1}
\;\big|\;
\bfz\in S^{2d+1}
\Big\}
=\{\bfz\cdot\bar{\bfz}=1\}
\,.
\]
Furthermore the composition
of the exhausting, strictly plurisubharmonic function $f$
with the inverse of the coordinate change
$\Phi^{-1}(\bfz,\bfw)=\frac12\big(\bfz+\bfw,\rmi(\bfw-\bfz)\big)$
yields the exhausting, strictly plurisubharmonic function 
$f\circ\Phi^{-1}\co\Phi(Q)\ra[0,\infty)$
given by the restriction of
\[
f\circ\Phi^{-1}\big(\bfx+\rmi\bfy,\bfu+\rmi\bfv\big)=
\frac18\big(|\bfy+\bfv|^2+|\bfx-\bfu|^2\big)
\]
to the image quadric $\Phi(Q)$.


\subsection{Projectisation of the second factor\label{subsec:projofsecfak}}

Denote the natural $\C^*$-action on $\C^{d+1}$
by $(a,\bfw)\mapsto a\bfw$.
The projection $\bfw\mapsto[\bfw]$
to the orbit space is a surjective submersion
$[\,.\,]\co\C^{d+1}\setminus\{0\}\ra\C P^d$.
Define
\[
\Psi:=(\id\oplus[\,.\,])\circ\Phi
\,,
\]
which is a surjective submersive holomorphic map
\[
\Big\{
(\bfz,\bfw)\in\C^{d+1}\times\C^{d+1}
\;\big|\;
\bfz\neq\rmi\bfw
\Big\}
\lra
\C^{d+1}\times\C P^d
\,.
\]
In coordinates, explicitly, we have
\[
\Psi(\bfz,\bfw)=
\big(\bfz+\rmi\bfw,[\bfz-\rmi\bfw]\big)
\,.
\]
As the domain of $\Psi$ contains $Q$,
the restriction
\[
\varphi:=\Psi|_Q\co Q\lra
\big(\C^{d+1}\times\C P^d\big)\setminus H
\]
is a well defined biholomorphism,
where
\[
H:=\Big\{
\big(\bfz,[\bfw]\big)\in\C^{d+1}\times\C P^d
\;\big|\;
\bfz\cdot\bfw=0
\Big\}
\,.
\]
In addition, the image 
\[
\Big\{
(\bfz,[\bar{\bfz}])\in\C^{d+1}\times\C P^d
\;\big|\;
\bfz\in S^{2d+1}
\Big\}
=\{\bfz\cdot\bar{\bfz}=1\}
\]
of the anti-Hopf map introduced in Section \ref{subsec:capandred}
does not intersect $H$.
Furthermore the composite
\[
\hat{f}:=f\circ\varphi^{-1}
\co
\big(\C^{d+1}\times\C P^d\big)\setminus H
\lra
[0,\infty)
\]
is exhausting, strictly plurisubharmonic function.





\subsection{Disc bundle structure\label{subsec:quaddiscbund}}

The total space of $TS^{2d+1}$
can be understood as the set of all
$(\bfq',\bfp')\in\R^{2d+2}\times\R^{2d+2}$
such that $|\bfq'|=1$ and $\bfq'\cdot\bfp'=0$.
Writing
\[
\bfq=\bfx+\bfu
\,,\quad
\bfp=\bfy+\bfv
\,,
\]
it is diffeomorphic to the quadric $Q$ via the inverse of
\[
Q
\lra
TS^{2d+1}
\,,
\quad
\bfq+\rmi\bfp\longmapsto
\Big(\frac{\bfq}{|\bfq|},-|\bfq|\bfp\Big)
\,.
\]
Therefore,
the sublevel sets of
$f(\bfq,\bfp)=\frac12|\bfp|^2$, $(\bfq,\bfp)\in Q$,
and, hence, of $\hat{f}$,
admit the structure of a tubular neighbourhood of
$\varphi\big(\{\bfp=0\}\big)$.



\subsection{Projectisation of the first factor\label{subsec:projoffirstfak}}

We consider the affine embedding
\[
\C^{d+1}\times\C P^d
\lra
\C P^1\times\C^d\times\C P^d
\]
given by
\[
\big(\bfz,[\bfw]\big)
\longmapsto
\Big([1:z_0],(z_1,\ldots,z_d),[\bfw]\Big)
\]
using coordinates $[z_0':z_0]$ on the $\C P^1$-factor.
Define a complex hypersurface
\[
\hat{H}\subset
\C P^1\times\C^d\times\C P^d
\]
by
\[
z_0w_0+z_0'\big(z_1w_1+\ldots+z_dw_d\big)=0
\,,
\]
which is the homogenisation of $\bfz\cdot\bfw=0$
corresponding to $[z_0':z_0]$.
Observe that the affine part equals
$H=\hat{H}\cap\{z_0'=1\}$.
On $\{z_0=1\}$ we read off the equation
$w_0+z_0'\big(z_1w_1+\ldots+z_dw_d\big)=0$.
Therefore,
\[
\hat{H}=H\cup\Big([0:1]\times\C^d\times\{w_0=0\}\Big)
\,,
\]
so that
\[
\big(\C P^1\times\C^d\times\C P^d\big)\setminus\hat{H}
\]
is equal to the union of
\[
\varphi(Q)=\big(\C^{d+1}\times\C P^d\big)\setminus H
\quad\text{and}\quad
\Big([0:1]\times\C^d\times\{w_0\neq0\}\Big)
\,.
\]
Precomposing the embedding obtained
with the biholomorphism $\varphi$
from Section \ref{subsec:projofsecfak}
yields a holomorphic embedding
\[
\hat{\varphi}\co
Q\lra
\big(\C P^1\times\C^d\times\C P^d\big)\setminus\hat{H}
\]
of the affine quadric $Q$.

In the intersection pattern in
Section \ref{subsec:complexcycles}
we substitute $H_2$ by $\hat{H}$.
We claim that
\[
C_i\cdot\hat{H}=1
\,,\quad
i=1,2
\,,
\]
for the intersection numbers.
Indeed, representing $C_1$ by
\[
\C P^1\ni[z_0':z_0]
\longmapsto
\Big([z_0':z_0],(0,\ldots,0),[1:0:\ldots:0]\Big)
\]
the intersection with $\hat{H}$ 
is determined by $z_0=0$ on $\{z_0'=1\}$;
on $\{z_0=1\}$ there is no intersection.
The class $C_2$ can be represented by
\[
\C P^1\ni[z:w]
\longmapsto
\Big([1:0],(0,\ldots,1),[0:\ldots:0:z:w]\Big)
\]
such that the intersection with $\hat{H}$
given by $w=0$.


\subsection{A Pl\"ucker--Segre--Veronese type embedding\label{subsec:apsvtypeemb}}

Introducing coordinates
\[
\Big([z_0':z_0],[z_1':z_1:\ldots:z_d],[\bfw]\Big)
\in
\C P^1\times\C P^d\times\C P^d
\]
we homogenise $\bfz\cdot\bfw=0$ a third time
and obtain as in Section \ref{subsec:projoffirstfak}
\[
z_0z_1'w_0+z_0'\big(z_1w_1+\ldots+z_dw_d\big)=0
\,.
\]
Writing this as
\[
z_{0b0}+z_{a11}+\ldots+z_{add}=0
\]
we get a linear equation in $d+1$ unknowns.
Organising this with help of homogeneous coordinates $z_{ijk}$
for $i,k=0,\ldots,d$ and $j=a,b$
we end up with an holomorphic embedding
\[
\C P^1\times\C^d\times\C P^d
\lra
\C P^1\times\C P^d\times\C P^d
\lra
\C P^N
\,,\quad
N=2(d+1)^2-1
\,,
\]
that maps $\hat{H}$ into a hyperplane of $\C P^N$.
After a linear coordinate change we assume that the hyperplane
containing the image of $\hat{H}$ is simply $\C P^{N-1}$.
This embedding restricts to a holomorphic embedding
\[
\psi\co
\big(\C P^1\times\C^d\times\C P^d\big)\setminus\hat{H}
\lra\C^N
\,,
\]
so that $\im\psi$ is Stein. 
The pull back of $\frac12|\,.\,|^2$
along $\psi$ defines an exhausting,
strictly plurisubharmonic function $\hat{g}$
on the complement of
$\hat{H}$ in $\C P^1\times\C^d\times\C P^d$.


\subsection{Extending a plurisubharmonic function\label{subsec:exapshf}}

We want to modify the exhausting,
strictly plurisubharmonic function
$f\co Q\ra[0,\infty)$ from Section \ref{subsec:anaffquad}.
For that we define a second exhausting,
strictly plurisubharmonic function
$g\co Q\ra[0,\infty)$ by $g:=\hat{\varphi}^*\hat{g}$,
see Sections \ref{subsec:projoffirstfak}
and \ref{subsec:apsvtypeemb}.

Choose positive real numbers $g_0<g_1<g_2$
such that $\{f=0\}$ is contained
in the relatively compact sublevel set $\{g<g_0\}$.
Denote by $\chi\co Q\ra[0,1]$ a smooth cut--off function
that is equal to $1$ on $\{g\leq g_1\}$ and
$0$ on $\{g\geq g_2\}$.
Let $\kappa\co\R\ra[0,\infty)$ be a smooth function
that is $0$ on $\{g\leq g_0\}$,
satisfies $\kappa'>0$, $\kappa''\geq0$ on $\{g>g_0\}$,
and is of the form $Ag+B$ on $\{g\geq g_2\}$
for positive constants $A,B\in\R$.
As in \cite{dvl91},
the function $h\co Q\ra[0,\infty)$ defined by
\[
h:=\chi f+C\kappa(g)
\]
is an exhausting, strictly plurisubharmonic function
for $C>0$ sufficiently large.
Indeed, $h=f$ on $\{g\leq g_0\}$.
By \cite[p.~19, Remark 2.8]{ce12},
the composition $\kappa(g)$ is strictly plurisubharmonic on $\{g>g_0\}$,
so that $h$ is strictly plurisubharmonic on $\{g\leq g_1\}$.
On the compact set $\{g_1\leq g\leq g_2\}$
we choose $C$ large enough
such that the Levi form of $C\kappa(g)$ dominates
the one of $\chi f$.
On $\{g\geq g_2\}$ we finally have
\[
h=C(Ag+B)
\,.
\]

Consequently,
$h\circ\varphi^{-1}\co\big(\C^{d+1}\times\C P^d\big)\setminus H\ra[0,\infty)$
extends to an exhausting, strictly plurisubharmonic function
\[
\hat{h}\co
\big(\C P^1\times\C^d\times\C P^d\big)\setminus\hat{H}
\lra[0,\infty)
\,,
\]
which equals $C(A\hat{g}+B)$ in a neighbourhood of
the hypersurface $\hat{H}$.
After choosing the constant $g_0$ appropriately
we can assume that $\hat{h}$ restricts to $\hat{f}$
from Section \ref{subsec:projofsecfak}
on the tubular neighbourhood $D_{\delta}T^*S^{2d+1}$ of
\[
\Big\{
(\bfz,[\bar{\bfz}])\in\C^{d+1}\times\C P^d
\;\big|\;
\bfz\in S^{2d+1}
\Big\}
\subset
\C P^1\times\C^d\times\C P^d
\]
chosen in Section \ref{subsec:capandred}.


\subsection{Deforming the filling\label{subsec:wwhomotopyequivalent}}

We continue the discussion from Section \ref{subsec:exapshf}.
Denote by $M'$ a level set of $\hat{f}$
that is contained in $D_{\delta}T^*S^{2d+1}$.
In view of Section \ref{subsec:quaddiscbund},
we find by uniqueness of tubular neighbourhoods
up to isotopy an isotopy
of $M'$ and $S_{\delta}T^*S^{2d+1}$
inside $D_{\delta}T^*S^{2d+1}$.
The trace of the isotopy does not intersect
the zero section $S^{2d+1}$;
and so does not $D_{\varepsilon}T^*S^{2d+1}$
for $\varepsilon\in(0,\delta)$ sufficiently small.
  
By a priori usage of $D_{\varepsilon}T^*S^{2d+1}$
(replacing $\delta$ by $\varepsilon$)
in the gluing construction in Section \ref{subsec:capandred}
we can assume that $M'$ and $S_{\varepsilon}T^*S^{2d+1}$
are isotopic via an isotopy with trace
not intersecting $\Int\big(D_{\varepsilon}T^*S^{2d+1}\big)$.
With isotopy extension we see
that the union $W'$ of $W$
and the cobordism bounded by
$S_{\varepsilon}T^*S^{2d+1}$ and $M'$ inside $\hat{Z}$
deformation retracts onto $W$.
In particular,
$W'$ as submanifold of $\big(\hat{Z},\hat{\Omega}\big)$
is symplectically aspherical.
\hfill Q.E.D.


\begin{ack}
  We thank Hansj\"org Geiges for ongoing discussions
  on the diffeomorphism type of symplectic fillings.
  We thank Myeonggi Kwon and Takahiro Oba for 
  their comments on the manuscript
  and Stefan Nemirovski for pointing out the source \cite{dvl91}
  to us.
  We thank the participants of the A5/C5-Seminar
  Florian Buck, Jan Eyll,
  Jonas Fritsch and Lars Kelling.
\end{ack}


\end{document}